\documentclass[11pt]{amsart}
\usepackage{etex}
\usepackage{diagbox}
\usepackage{array}
\usepackage{graphicx, pictex}
\usepackage{stmaryrd}
\usepackage{changepage}
\usepackage{float}
\restylefloat{table}
\usepackage{multirow}
\usepackage[dvipsnames]{xcolor}
\usepackage{amsmath,amssymb,amsthm,mathrsfs}
\allowdisplaybreaks
\usepackage{yfonts}
\renewcommand{\d}{\partial}

\def\C{{\mathcal{C}}}

\def\d{\Omega}

\def\du#1#2#3{\overset{#3}{\underset{#2}{#1}}}
\def\Forall{\quad \hbox{ for all }}

\def\M{{\mathcal{M}}}

\newcommand{\tn}[1]{\lVert\kern-1pt\lvert{#1}\rvert\kern-1pt\rVert}
\def\<{{\langle}}
\def\>{{\rangle}}
\def\Forall{\quad \hbox{ for all }}

\def\C{{\mathcal{C}}}

\def\d{\Omega}

\def\Forall{\quad \hbox{ for all }}

\def\d{\Omega}

\def\Forall{\quad \hbox{ for all }}

\def\tb#1{{\|\kern-1pt| #1 \|\kern-1pt|}}
\def\nm2#1#2{\|#1\|_{2,\d_{#2}}}

\def\R{\mathbb{R}}
\def\P{\mathbb{P}}

 \theoremstyle{plain}
 \newtheorem{thm}{Theorem}[section]
 \numberwithin{equation}{section} %% Comment out for sequentially-numbered
 \numberwithin{figure}{section} %% Comment out for sequentially-numbered
 \theoremstyle{plain}
 
 \theoremstyle{plain}
 \theoremstyle{plain}
 \newtheorem{theorem}[thm]{Theorem}
 \theoremstyle{plain}
 \newtheorem{corollary}[thm]{Corollary}
\theoremstyle{plain}
 \newtheorem{remark}[thm]{Remark}
 \theoremstyle{plain}
 
\def\C{{\mathcal{C}}}

\def\M{{\mathcal{M}}}

\def\d{{\Omega}}

\def\Forall{\quad \hbox{ for all }}
\def\<{{\langle}}
\def\>{{\rangle}}
\def\R{\mathbb{R}}

\def\du#1#2#3{\overset{#3}{\underset{#2}{#1}}}

\usepackage{amssymb}

\begin{document}

\title[Connections Between  FD and FE  for Convection-Diffusion]
{Connections  Between Finite Difference  and Finite Element Approximations for  a Convection-Diffusion Problem}

\author{Constantin Bacuta}
\address{University of Delaware,
Mathematical Sciences,
501 Ewing Hall, 19716}
\email{bacuta@udel.edu}

\author{Cristina Bacuta}
\address{University of Delaware,
Mathematical Sciences,
501 Ewing Hall, 19716}
\email{crbacuta@udel.edu}


\thanks{The work was supported  by NSF-DMS 2011615}%

%35K57 Reaction-diffusion equations, 
%35Q35 PDEs in connection with fluid mechanics
%65F(1973–now) Numerical linear algebra
%65H10(1973–now)Numerical computation of solutions to systems of equations
%65N06 (1991–now)Finite difference methods for boundary value problems involving PDEs
%65N12 (1991–now)Numerical solution of discretized equations for boundary value problems involving PDEs 
%65N22 (1991–now)Numerical solution of discretized equations for boundary value problems involving PDEs 
%65N30(1973–now)Finite element, Rayleigh-Ritz and Galerkin methods for boundary value problems involving PDEs
%74S05(2000–now)Finite element methods applied to problems in solid mechanics
%76(1940–now)Fluid mechanics For general continuum mechanics, see 74Axx, or other parts of 74-XX
%76R(1980–now)Diffusion and convection
%76R10 Free convection
%76R50 Diffusion

\begin{abstract} We consider a model convection-diffusion problem and \\ present useful connections between the finite differences and finite element discretization methods. We  introduce a general upwinding Petrov-Galerkin discretization based on  bubble modification of the test space and connect the method with the general upwinding approach used in finite difference discretization. We  write the finite difference and the  finite element  systems such that  the two corresponding linear systems have the same stiffness matrices, and compare the right hand side load vectors for the two methods. This new approach allows for improving well known upwinding finite difference methods and for obtaining new error estimates. We prove that  the exponential bubble Petrov-Galerkin discretization can recover the interpolant of the exact solution. As a consequence, we estimate the closeness of the related finite difference solutions  to the interpolant. The ideas we present in this work, can lead to building efficient new discretization methods for multidimensional convection dominated problems. 
%Using the concept of optimal trial norm we  provide  new error estimates  based on the connection between the two methods. 

\end{abstract}
\maketitle

%%%%%%%%%%%%%%%%%            Introduction 
\section{Introduction}

%bpvscale
%$\aleph_0$, \textfrak{c}
We start with the model of a  singularly perturbed  convection diffusion problem: Find $u=u(x)$  on $[0, 1]$ such that
\begin{equation}\label{eq:1d-model}
\begin{cases}-\varepsilon u''(x)+ \kappa\, u'(x)=f(x),& 0<x<1\\ u(0)=0, \ u(1)=0,\end{cases}
\end{equation} 
where $\varepsilon$ and $\kappa$ are positive constants. In this paper, we consider the convection dominated case, i.e.  $\varepsilon \ll 1$. Here, the function $f$ is given and assumed to be continuous  on $[0, 1]$.  Without loss of generality, we  will further  assume that $\kappa=1$. 

The model problem \eqref{eq:1d-model} and its multi-dimensional  variants arise  when solving   heat transfer problems in thin domains, as well as when using small step sizes in implicit time discretizations of parabolic convection diffusion type problems, see \cite{Lin-Stynes12}. The solutions to these problems are characterized by  boundary layers,  see e.g., % of width $\mathcal{O}(\varepsilon^{1/2}\ln(1/\varepsilon))$ 
\cite{Dahmen-Welper-Cohen12, EEHJ96,  linssT10, Roos-Schopf15, zienkiewicz2014}. 
Approximating such solutions poses  numerical  challenges due to the $\varepsilon$-dependence of the  stability constants and  of the  error estimates.  
There is a tremendous amount of literature addressing  these types of problems, see e.g.  \cite{EEHJ96, linssT10, quarteroni-sacco-saleri07, Roos-Schopf15, zienkiewicz2014}. 
The goal of this paper is to use  connections between upwinding Finite Differences (FD) and certain Finite Element (FE)  discretizations of the model convection diffusion problem  \eqref{eq:1d-model}, to improve the performance of the upwimding FD methods, and to find new error  estimates  for both methods. %The results  can be used to  build  efficient new methods for the multidimensional version of convection dominated problems. 
%dem-fuh-heu-tia19
We introduce a general upwinding FE Petrov-Galerkin (PG) discretization based on  bubble modification of the test space. The test space is modified by  using translations of a generating bubble function. 

For studying stability error estimates and connection with FD methods, we use the  concept of optimal trial norm, as presented in \cite{BHJ21, BHJ22, Broersen-StevensonDPGcd14, Chan-Heuer-Bui-Demkowicz14, Dahmen-Welper-Cohen12, demkowicz-gopalakrishnanDPG10, Barrett-Morton81}. We write the finite difference and the finite element systems  for uniformly distributed nodes such that the two corresponding linear systems have the same stiffness matrices, and compare the Right Hand Side (RHS)  load vectors for the two methods. The same technique is applied for the corresponding variational formulations of the FE and FD methods by finding a common bilinear form and by comparing the RHS functionals.  We emphasize that any upwinding FD method can be deduced from an FE  PG method by carefully selecting the generating bubble function and a quadrature to approximate the RHS dual vector of the FE  PG system. The approach allows for improving the performance of known upwinding FD approaches.  %connect the method with the general upwinding approach used in finite difference discretization. 
In addition, we investigate a particular PG method based on an exponential  generating bubble function and prove that the method recovers the interpolant of the exact solution. This leads to new error estimates for the corresponding upwinding FD method.    

The rest of the paper is organized as follows. We review the upwinding FD method in Section \ref{sec:FD}, and  the FE discretization together with the concept of optimal trial space in Section \ref{sec:FE}. We introduce a general upwinding Petrov-Galerkin discretization method  and relate it with the upwinding FD method in Section \ref{sec:PG}. In Section  \ref{sec:Quad-Bubbles} and Section \ref{sec:Exponential-Bubbles}, we define particular test spaces based on quadratic bubbles and exponential type bubbles, respectively, and  connect the new PG methods with known upwinding FD methods. 
%We conclude in section  Section \ref{sec:conclusion}.  
%%%%%%%%%%%%%%%%%%%%%
%SECTION
%%%%%%%%%%%%%%%%%%%%%
\section{Standard Finite Difference Discretization}\label{sec:FD} In this section, we review the standard upwinding 
FD discretization of  \eqref{eq:1d-model} on $[0,1]$ on uniform meshes.   We divide the interval into $n$  subintervals using the   uniformly distributed nodes  $x_j=h j$, with $j=0,1, \ldots, n$ and $h=\frac{1}{n}$, and consider the second order  finite difference approximation for $u''(x_j)$ and $u'(x_j)$ at the nodes $x_{j-1}, x_j$, and $x_{j+1}$, to obtain the linear system:

 \begin{equation}\label{FDs1}
  \left\{
 \begin{array}{rcl}
     u_0 & =0&\\ \\
 \displaystyle  -\varepsilon  \frac{ u_{j-1} - 2 u_j + u_{j+1}} {h^2 } + \frac {-u_{j-1} + u_{j+1}}{2h} & =f_j& \ j=\overline{1,n-1},\\  \\
     u_n & =0,&
\end{array} 
\right. 
\end{equation}
where $f_j=f(x_j)$.  Multiplying the generic equation in \eqref{FDs1} by $h$,  gives
 \begin{equation}\label{FDs2}
   \left\{
 \begin{array}{rcl}
     u_0 & =0&\\ 
 \displaystyle  \varepsilon  \frac{-u_{j-1} + 2 u_j - u_{j+1}} {h} + \frac {-u_{j-1} + u_{j+1}}{2} & =h\, f_j& \ j=\overline{1,n-1}.\\  
     u_n & =0.&
\end{array} 
\right. 
\end{equation}
Since the convection coefficient  $\kappa=1>0$,  the standard  FD {\it  upwinding method} for  discretizing  \eqref{eq:1d-model}, requires the {\it backward difference method} for approximating $u'(x_j)$ and leads  to the system 

 \begin{equation}\label{FDuw1}
  \left\{
 \begin{array}{rcl}
     u_0 & =0&\\ 
  \displaystyle  \varepsilon  \frac{-u_{j-1} + 2 u_j - u_{j+1}} {h} + ( -u_{j-1} + u_{j}) & =h\, f_j& \ j=\overline{1,n-1}.\\  
       u_n & =0.&
\end{array} 
\right. 
\end{equation}

 Using that 
\[
 u_{j} -u_{j-1} = \frac {-u_{j-1} + u_{j+1}}{2} + \frac{-u_{j-1} + 2 u_j - u_{j+1}} {2},   
\]
the system \eqref{FDuw1} becomes 
 \begin{equation}\label{FDUW1}
   \left\{
 \begin{array}{rcl}
     u_0 & =0&\\ 
 \displaystyle  \varepsilon \left (1 +\frac{h}{2 \varepsilon} \right)\frac{-u_{j-1} + 2 u_j - u_{j+1}} {h} + \frac {-u_{j-1} + u_{j+1}}{2} & =h\, f_j&\\  
     u_n & =0.&
\end{array} 
\right. 
\end{equation}
The quantity $\frac{h}{2 \varepsilon}$ is known as {\it the local Peclet number} and is denoted by 
\[
\P e=\frac{h}{2 \varepsilon}.
\] 
According to Section 12.5 of \cite{quarteroni-sacco-saleri07},  the upwinding scheme \eqref{FDuw1} can be viewed as a centered difference scheme for the convection term with a correction for the coefficient of the  diffusion term  by $\varepsilon \cdot \P e$. The correction process is known as adding  {\it  artificial diffusion} or  
 {\it numerical viscosity}. As presented in \cite{quarteroni-sacco-saleri07, roos-stynes-tobiska-96}, more general  discretization based on {\it artificial diffusion} can be written as follows: 
  \begin{equation}\label{FDUWgeneral}
   \left\{
 \begin{array}{rcl}
     u_0 & =0&\\ 
 \displaystyle  \varepsilon_h \ \frac{-u_{j-1} + 2 u_j - u_{j+1}} {h} + \frac {-u_{j-1} + u_{j+1}}{2} & =h\, f_j&\\  
     u_n & =0,&
\end{array} 
\right. 
\end{equation}
where, for a smooth function $\Phi: (0, \infty) \to   (0, \infty)$ with $\displaystyle \lim_{t\to 0}\Phi(t)=0$, 
\[
 \varepsilon_h = \varepsilon \left (1 + \Phi(\P e) \right ).
\]
We note  that $\Phi(\P e) =\P e$ corresponds to the standard upwinding method \eqref{FDUW1}  or \eqref{FDuw1}. 

For the general upwinding case,  from \eqref{FDUWgeneral},  we obtain the system
 \begin{equation}\label{1d-LSfd}
 \left (\frac{\varepsilon_h}{h}  S+ C \right )\, U =  F_{fd}, 
\end{equation}
where \(U,F\in\R^{n-1}\) and \(S, C \in\R^{(n-1)\times(n-1)}\) with:
 \[
 U:=\begin{bmatrix}u_1\\u_2\\\vdots\\u_{n-1}\end{bmatrix},\quad F_{fd}:= h\, \begin{bmatrix} f(x_1)\\ f(x_2)\\\vdots \\ f(x_{n-1})\end{bmatrix}\, \text{and} 
\]
 \begin{equation}\label{eq:SC}
 S:=\begin{bmatrix}2&-1\\-1&2&-1\\&\ddots&\ddots&\ddots\\&&-1&2&-1\\&&&-1&2\end{bmatrix},\quad 
 C:=\frac{1}{2}\begin{bmatrix}0&1\\-1&0&1\\&\ddots&\ddots&\ddots\\&&-1&0&1\\&&&-1&0\end{bmatrix}.
 \end{equation}
 Using the ``tridiagonal notation'', we have that
 \[
S=tridiag(-1, 2, -1),   \ C= tridiag\left (-\frac12, 0, \frac12 \right ), 
 \]
and the matrix of the finite difference system \eqref{1d-LSfd} is 
\begin{equation} \label{eq:M4CDfd}
M_{fd}= tridiag\left ( -\frac{\varepsilon_h}{h} - \frac{1}{2},\  \frac{2\, \varepsilon_h}{h},\  -\frac{\varepsilon_h}{h} + \frac{1}{2} \right ).
\end{equation}

In Section \ref{sec:PG}, we will see that the general FD discretization  \eqref{FDUWgeneral}, leading to \eqref{1d-LSfd}, relates with a Petrov-Galerkin method with a bubble test space.  

%%%%%%%%%%%%%%%%%%%
%SECTION  boffi-brezzi-demkowicz-duran-falk-fortin2006,
%%%%%%%%%%%%%%%%%%%
\section{Finite element  linear variational formulation and discrete optimal trial norm}\label{sec:FE}
For the finite element discretization, we  will use the following  notation:
\[ 
\begin{aligned}
a_0(u, v) & = \int_0^1 u'(x) v'(x) \, dx, \ (f, v) = \int_0^1  f(x) v(x) \, dx,\ \text{and}\\
b(v, u)& =\varepsilon\, a_0(u, v)+(u',v)  \ \text{for all} \ u,v \in V=Q=H^{1}_0(0,1).
\end{aligned}
 \]
A variational formulation of \eqref{eq:1d-model}, with $\kappa=1$, is: \\ Find $u \in Q:= H_0^1(0,1)$ such that
 \begin{equation}\label{VF1d}
b(v,u) = (f, v), \ \text{for all} \ v \in V=H^{1}_0(0,1).
\end{equation}
The existence and uniqueness of the solution of \eqref{VF1d} is well known, see e.g., \cite{bartels16, boffi-brezzi-fortin, braess, brenner-scott, demko23, ern-guermond, demko-oden}.

\subsection{Standard discretization with $C^0-P^1$ test and trial spaces}\label{sec:1d-lin-discrete}

We  divide the interval $[0,1]$ into $n$  equal length subintervals using the nodes $0=x_0<x_1<\cdots < x_n=1$ and denote  $h:=x_j - x_{j-1}=1/ n$. For the above uniformly distributed nodes, we define  the corresponding finite element discrete space  $\M_h$  as  the subspace of $Q=H^1_0(0,1)$, given by
 \[ 
 \M_h = \{ v_h \in V \mid v_h \text{ is linear on each } [x_j, x_{j + 1}]\},
 \]
  i.e., $\M_h$ is the space of all {\it continuous piecewise linear  functions} with respect to the given nodes, that {\it are zero at $x=0$ and $x=1$}.  We consider the nodal basis $\{ \varphi_j\}_{j = 1}^{n-1}$ with the standard defining property $\varphi_i(x_j ) = \delta_{ij}$.  
We couple the above discrete trial space with the  discrete  test space $V_h:=\M_h$.  
 Thus, the discrete  variational formulation of \eqref{VF1d} is: Find $u_h \in \M_h$ such that
 \begin{equation}\label{dVF}
b(v_h, u_h) = (f, v_h), \ \text{for all} \ v_h \in V_h.
\end{equation}
We look for  $u_h \in V_h$ with  the nodal basis expansion
 \[
 u_h := \sum_{i=1}^{n-1} u_i \varphi_i, \ \text{where} \ \ u_i=u_h(x_i).
 \]
If we consider the test functions $v_h=\varphi_j, j=1,2,\cdots,n-1$ in \eqref{dVF}, we  obtain  the following linear system 
 \begin{equation}\label{1d-LS}
 \left (\frac{\varepsilon}{h}  S+ C \right )\, U = F_{fe}, 
\end{equation}
where
 \[
 U:=\begin{bmatrix}u_1\\u_2\\\vdots\\u_{n-1}\end{bmatrix},\quad F_{fe}:=\begin{bmatrix}(f,\varphi_1)\\ (f,\varphi_2)\\\vdots \\ (f,\varphi_{n-1})\end{bmatrix}, 
 \]
and  \(S, C \in\R^{(n-1)\times(n-1)}\) are given in \eqref{eq:SC}. 
 %%%%%%%%%%%%%%
 %SUBSECTION
 %%%%%%%%%%%
 \subsection{Optimal discrete  trial norms } \label{sec:optimalnorm-linears}
%%%%%%%%%%%%%%%
For  $V=Q=H^1_0(0,1)$, we consider the  standard inner product  given by  $a_0(u,v) = (u,v)_V = (u',v')$,  and let  $\M_h, V_h $  be the standard space of continuous piecewise linear functions 
\[
\M_h=V_h=span \{\varphi_1, \cdots,\varphi_{n-1}\}. 
\]
For the purpose of error analysis, on $V$ and  $V_h$, we consider the standard norm induced by $a_0(\cdot,\cdot)$, i.e,  $ |v|^2:=a_0(v,v)$, but on $\M_h$, we will introduce a different norm. %, optimal  from the stability point of view. 
On $V_h\times \M_h$,  we define the bilinear form
 \begin{equation}\label{dVFad}
b_d(v_h, u_h)=  d\, a_0(u_h, v_h)+(u'_h,v_h)\  \text{for all} \ u_h\in \M_h, v_h \in V_h,
\end{equation}
where $d=d_{\varepsilon,h}$ is a constant, that might depend on $h$ and $\varepsilon$ and is motivated by the 
{\it upwinding Petrov-Galerkin method} introduced in Section \ref{sec:GenBubble}. 
 
 The discrete optimal trial norm   on $ \M_h$ is defined by 
 \begin{equation}\label{eq:dotn}
  \|u_h\|_{*,h}:= \du {\sup} {v_h \in V_h}{} \ \frac {b_d(v_h, u_h)}{|v_h|}.
 \end{equation}
  An explicit representation of this norm is established in \cite{CRD-results}. We have

\begin{equation}\label{eq:COptNorm-h} 
\|u_h\|_{*,h}^2  = d^2|u_h|^2 +|u_h|^2_{*,h}, 
\end{equation}
where, 
 \begin{equation}\label{eq:PhTu-Norm} 
|u|^2_{*,h}= \frac{1}{n}\sum_{i=1}^n \left(\frac{1}{h}\, \int_{x_{i-1}}^{x_i} u(x)\,dx\right)^2 - \left(\int_0^1u(x)\,dx\right)^2.
\end{equation}
We note that $|\cdot |_{*,h}$ is a  semi-norm on $V_h$, since we can have $|u_h|_{*,h}= 0$ for  any non zero function $u_h \in \M_h$ such that 
\[
  \int_{x_{i-1}}^{x_i} u_h (x)\,dx = \frac{1}{n} \, \int_0^1u_h(x)\,dx, \ i=1,2,\cdots,n.
\]
In particular, for $n=2m$, such a function is $u_h =\varphi_1 + \varphi_3+\cdots + \varphi_{2m-1}$.
Using  the Cauchy-Schwarz inequality, we  can also check that $|u|_{*,h} \leq \|u\|$. Indeed, 
\begin{equation}\label{weekN}
|u|^2_{*,h} \leq  \frac{1}{n}\sum_{i=1}^n \left(\frac{1}{h} \int_{x_{i-1}}^{x_i} u(x)\,dx\right)^2 \leq 
\sum_{i=1}^n \int_{x_{i-1}}^{x_i} u^2(x)\,dx=\|u\|^2. 
\end{equation}
The above estimates together with the formula \eqref{eq:COptNorm-h} suggest that the  discrete optimal trial norm  $\|\cdot\|_{*,h}$ could be a weak norm on $\M_h$ if $d$ is very small. 

The optimal test norm, helps with stability estimates in the following sense. If we consider the problem: 
Find $u_h \in \M_h$ such that
 \begin{equation}\label{dVFd}
b_d(v_h, u_h) = F_h(v_h), \ \text{for all} \ v_h \in V_h=\M_h,
\end{equation}
where $F_h: V_h \to \R$ is  a linear functional on $V_h$, then according to the definition \eqref{eq:dotn}, we have
\begin{equation}\label{eq:OptD}
  \|u_h\|_{*,h}:= \du {\sup} {v_h \in V_h}{} \ \frac {b_d(v_h, u_h)}{|v_h|} = \du {\sup} {v_h \in V_h}{} \ \frac {F_h(v_h)}{|v_h|} :=\|F_h\|_{V_h^*}.
 \end{equation}
 \begin{remark} \label{rem:u2-u1}
 In particular, if $F_1$ and $F_2$ are two linear functionals on $V_h$ and we solve for  $u_h^1, u_h^2 \in \M_h$ such that
 \begin{equation}\label{dVF12}
 \begin{aligned}
b_d(v_h, u_h^1)  & = F_1(v_h), \ \text{for all} \ v_h \in V_h=\M_h,  \text{and} \\
 b_d(v_h, u_h^2) & = F_2(v_h),  \ \text{for all} \ v_h \in V_h=\M_h, 
\end{aligned} 
\end{equation}
then 
 \begin{equation}\label{eq:OptEr}
  \|u_h^2-u_h^1\|_{*,h}= \|F_2 -F_1\|_{V_h^*}.
 \end{equation}
 As an application, we can have $u_h^1$  be the FD approximation and $u_h^2$ be the FE approximation  of  problem \eqref{eq:1d-model}. In this case, we can estimate the difference between the two solutions in the  $\|\cdot \|_{*,h}$ norm by finding an upper bound for $ \|F_2 -F_1\|_{V_h^*}$.
 \end{remark} 
%On the other hand, by rearranging  the integrals in  \eqref{eq:PhTu-Norm}, we obtain 
%\[|u|^2_{*,h} = \frac{\sum_{i=1}^n \overline{u_i}^2}{n} - \left ( \frac{\sum_{i=1}^n \overline{u_i}}{n}  \right )^2,\]
%where $\overline{u_i}:= \frac{1}{h} \int_{x_{i-1}}^{x_i} u(x)\,dx$. This shows that  $|\cdot |^2_{*,h}$ is in general a seminorm
%and we can have $|u|_{*,h} =0$ {\it if and only if}
%\[ \int_{x_{i-1}}^{x_i} u(x)\,dx= \frac{1}{n} \int_0^1u(x)\,dx, \ \text{ for all } i=1,2,\cdots,n.\]
%If $u=w_h\in \M_h$, the above condition can be  satisfied for $n=2m$ and  $
%w_h=C \sum_{i=1}^m \varphi_{2i-1}$, where $C$ is an arbitrary constant.  The graph of  such $w_h$ is highly oscillatory when $h\to 0$ and $|w_h|_{*,h}=0$.  
%The numerics show that, for small $\varepsilon/h$, the linear solution  can capture the mode $w_h$ leading  to 
%The ``zig-zag'' behavior of the linear  finite  element solution $u_h$ for small $\varepsilon/h$ can be  justified also   
%by noting that $CU=F$ for $U$ the coefficient vector of $w_h$, see  \eqref{1d-LS}. Thus, the solution $u_h$  can capture the oscillatory  mode $w_h$  or is  ``insensitive'' to perturbation by $c\, w_h$.   
%SUB-SECTION
 
 %%%%%%%%%%%%%%%
%SECTION
%%%%%%%%%%%roos-stynes-tobiska-96}, or to  the {\it Scharfetter-Gummel (SG) method}, according to \cite{quarteroni-sacco-saleri07
\section{The Petrov-Galerkin method  with bubble type test space } \label{sec:PG}
%%%%%%%%%%%%%%%
For improving the stability and approximability of the standard linear finite element approximation for solving \eqref{VF1d}, various  Petrov-Galerkin discretizations  were considered, see e.g., \cite{CRD-results, zienkiewicz76, mitchell-griffiths, roos-stynes-tobiska-96, zienkiewicz2014}. In this section, we introduce a general class of upwinding PG discretizations based on a bubble modification of the  standard $C^0-P^1$ test space.  
The idea is to define  $V_h$  by adding to each $\varphi_j$, a pair of polynomial bubble functions.
According to Section 2.2.2 in \cite{zienkiewicz2014}, this idea  was first suggested in \cite{ZGHupwind76} and used in the same year in \cite{zienkiewicz76} with  quadratic bubble modification.  
%The choice of  $V_h$ defined by adding to each $\varphi_j$, a pair of polynomial bubble functions of  
The method  is known in literature as {\it upwinding PG method}, according to Section 2.2.2 in \cite{zienkiewicz2014}, or {\it upwinding finite element method}, according to Section 2.2 in \cite{roos-stynes-tobiska-96}.

%The  PG method that can be found in  \cite{CRD-results, mitchell-griffiths},  and uses quadratic bubble upwinding. 
Besides a more general approach of the method, we discover an equivalent variational  reformulation of the proposed PG  discretization that uses  a new bilinear form defined on {\it standard  linear finite element spaces}. The   new formulation  leads to strong connections with the upwinding FD methods and to a better understanding of both the FD and  the FE  methods. 

The standard  variational formulation for solving  \eqref{eq:1d-model} with $\kappa=1$, is: Find $u \in Q=H^{1}_0(0,1)$ such that 
\begin{equation}\label{PG4model}
b(v, u) =\varepsilon\, a_0(u, v)+(u',v) = (f,v ) \ \Forall  v \in V=H^{1}_0(0,1).\\
\end{equation}
%Only for the analysis purpose, we involve  an optimal norm on $Q$.  
A general  Petrov-Galerkin method for solving \eqref{PG4model} chooses  a test space $V_h \subset V=H^{1}_0(0,1)$ that is different from the trial space $\M_h \subset Q=H^1_0(0,1)$. 
%and $V_h \subset V=H^1_0(0,1)$ and solve the discrete problem: 
%Find $ u_h \in \M_h$ such that 
%\begin{equation}\label{PGmodel-h}
%b( v_h, u_h) = (f,v_h ) \ \Forall  v_h \in V_h.
%\end{equation}

%SUB-SECTION
%%%%%%%%%%%%%%%
%new subsection. Typed by Cristina
%\subsubsection{Bubble Functions}
\subsection{General Bubble Upwinding Petrov-Galerkin Method}\label{sec:GenBubble} 

On $[0,h]$, consider a continuous  bubble generating function $B:[0,h] \to\R$ with the following properties:
\begin{equation}\label{Bbounds}
B(0)=B(h)=0,\\
\end{equation}

\begin{equation}\label{b1}
\int_0^h B(x)  \, dx=b_1h \ \text{with} \ b_1>0.\\
\end{equation}

%\begin{equation}\label{b2}
%\int_0^h (B'(x))^2  \, dx=\frac{b_2}{h} \ \ \ (b_2>0).\\
%\end{equation}
 By translating $B$, we generate $n$ bubble functions that are locally supported. For $ i=1,2, \cdots, n$,  we define $B_i:[0, 1] \to \R$ by $B_i(x)=B(x-x_{i-1})=B(x-(i-1)h)$ on $[x_{i-1}, x_i]$, and we extend it  by zero to the entire interval $[0, 1]$. Note that $B_1=B$ on $[0, h]$, and for $ i=1,2, \cdots, n$, we have
\begin{equation}\label{B_i_bounds}
B_i(x_{i-1})=B_i(x_i)=0, \  \text{and} \ B_i=0\ \text{on} \ [0, 1]\backslash (x_{i-1}, x_i).
\end{equation}
In addition,
\begin{equation}\label{B_i_b1}
\int_{x_{i-1}}^{x_i} B_i(x)  \, dx=b_1h \ \text{with} \ b_1>0.\\
\end{equation}

%\begin{equation}\label{B_i_b2}
%\int_{x_{i-1}}^{x_i} (B_i ' (x) )^2 \, dx=\frac{b_2}{h}.\\
%\end{equation}
Next, we consider a particular class of Petrov-Galerkin discretizations of the model problem \eqref{PG4model} with trial space $\M_h= span\{ \varphi_j\}_{j = 1}^{n-1}$ and the test space  $V_h$ obtained by modifying $M_h$  such that diffusion is created from the convection term. 
%We could not determined  who first introduced a PG discretization based on this idea, but  the earliest  work  we could  find mentioning  {\it upwinding finite element}, is  the book by Mitchell and Griffiths, \cite{mitchell-griffiths}, where a basis for $V_h$ is defined by adding to each $\varphi_j$, a pair of polynomial bubble functions. 
%We will recover in particular Mitchell and Griffiths' approach in Section  \ref{sec:Quad-Bubbles}. 

The FE {\it bubble upwinding}  idea is based on building $V_h$  by translating a general function $B$ satisfying \eqref{Bbounds} and \eqref{b1}. 
%special pair of local  bubble functions.  
% This is also known as an {\it upwinding finite element scheme}, see Section 2.2 in \cite{roos-stynes-tobiska-96}. We define the test space $V_h$ by introducing a  bubble function for each interval $[x_{i-1}, x_i], i=1,2, \cdots, n$:
% \[ B_i:= 4\, \varphi_{i-1}\, \varphi_i, \ \ i=1,2, \cdots, n, \] which is supported in $[x_{i-1}, x_i]$. The discrete test space $V_h$ is 
To be more precise, we define the test space $ V_h$ by 
 \[
 V_h:= span \{ \varphi_j  + (B_{j}-B_{j+1})\}_{j = 1}^{n-1}, 
 \]
 where $\{ B_i\}_{i=1,\cdots,n}$ are defined above  and satisfy  \eqref{B_i_bounds} and \eqref{B_i_b1}. 
We note that both $\M_h$ and $V_h$ have the same dimension of $(n-1)$. 

%In a more general approach the test functions can be defined using  upwinding parameters $\sigma_i >0$ to get $V_h:= span \{ \varphi_j  + \sigma_i( B_{j}-B_{j+1})\} _{j = 1}^{n-1}$. 

%\subsubsection{Variational formulation  and matrices}
The upwinding Petrov Galerkin discretization with general bubble functions for 
\eqref{eq:1d-model} is: Find $u_h \in \M_h$ such that 
\begin{equation}\label{eq:1d-modelPG}
b(v_h, u_h) = \varepsilon\, a_0(u_h, v_h)+(u'_h,v_h) =(f,v_h) \ \Forall  v_h \in V_h. 
\end{equation}
Next, we show that the variatonal formulation \eqref{eq:1d-modelPG} admits a reformulation that uses  a new bilinear form defined on {\it standard  linear finite element spaces}. We look for 
\[
u_h= \sum_{j=1}^{n-1} \alpha_j \varphi_j,
\]
and consider a generic test function 
\[
v_h= \sum_{i=1}^{n-1} \beta_i \varphi_i + \sum_{i=1}^{n-1}  \beta_i (B_i - B_{i+1}) = \sum_{i=1}^{n-1} \beta_i \varphi_i + \sum_{i=1}^{n}  (\beta_i - \beta_{i-1}) B_{i},
\]
where, we define $\beta_0=\beta_n=0$. By introducing the notation 
\[
B_h:=\sum_{i=1}^{n}  (\beta_i - \beta_{i-1}) B_{i},  \ \text{and } \  w_h:=  \sum_{i=1}^{n-1} \beta_i \varphi_i ,
\]
we get 
\[
v_h=w_h + B_h.
\]

By using the formulas \eqref{B_i_bounds}, \eqref{B_i_b1}, and the facts that $u'_h, w'_h $  are constant  on each of  the intervals $[x_{i-1}, x_i]$, and that $w'_h= \frac{\beta_i -\beta_{i-1} }{h}$ on $[x_{i-1}, x_i]$,  we obtain
\[
(u'_h, B_h) = \sum_{i=1}^{n}\int_{x_{i-1}}^{x_i}  u'_h (\beta_i - \beta_{i-1}) B_{i}=
 \sum_{i=1}^{n} u'_h \,  w'_h \int_{x_{i-1}}^{x_i}  B_{i} =b_1h \sum_{i=1}^{n} \int_{x_{i-1}}^{x_i}  u'_h  w'_h. 
\]
Thus
\begin{equation} \label{eq:upBh}
(u'_h, B_h) =b_1h (u'_h, w'_h), \ \text{where} \  v_h=w_h + B_h.
\end{equation}
In addition, since $u'_h$ is constant on $[x_{i-1},x_i]$, we have
\[
(u'_h, B'_i) =u'_h\int_{x_{i-1}}^{x_i} B'_i(x)  \, dx=\ 0 \ \text{for all} \ i=1, 2, \cdots,  n. 
\]
Hence, 
\begin{equation} \label{eq:upBph}
 (u'_h, B'_h) =0, \  \text{for all} \  u_h \in \M_h, v_h=w_h + B_h \in V_h.
\end{equation}
From  \eqref{eq:upBh} and  \eqref{eq:upBph}, for any $u_h \in \M_h$ and  $v_h=w_h + B_h \in V_h$, we get
\begin{equation} \label{eq:bPG}
 b(v_h, u_h) = \left (\varepsilon + b_1h\right )  (u'_h, w'_h) +  (u'_h, w_h).
\end{equation}
 The addition of the bubble part to the test space leads to the extra diffusion term $b_1h  (u'_h, w'_h)$ with  $b_1h >0$  matching the sign of the coefficient of $u'$ in  \eqref{eq:1d-model}. This justifies the terminology  of %{\it bubble upwinding} 
 {\it upwinding PG} method. 
 
 Here are two  important notes regarding the variational formulation \eqref{eq:bPG}. First, note that only the linear part $w_h$ of $v_h$ appears in the expression of $ b(v_h, u_h)$ of \eqref{eq:bPG}.  Second, note that  the  functional  $v_h \to (f, v_h)$ can  be also viewed as a functional only of the linear part $w_h$. Indeed, using the splitting $v_h=w_h + B_h $  with $\displaystyle B_h:=\sum_{i=1}^{n}  (\beta_i - \beta_{i-1}) B_{i}$, we have
\[
(f, v_h) =(f, w_h) + (f, \sum_{i=1}^{n}  h w'_h B_i) =(f, w_h) +h\,  (f, w'_h  \sum_{i=1}^{n}   B_i). 
\]
As a consequence, the variational formulation of the upwinding  Petrov-Galerkin method can be reformulated as:   Find $u_h \in \M_h$ such that 
\begin{equation}\label{eq:1d-modelPGR}
 \left (\varepsilon + b_1h\right )  (u'_h, w'_h) +  (u'_h, w_h) = (f, w_h )  + h\,  (f, w'_h  \sum_{i=1}^{n}   B_i),  w_h \in M_h. 
\end{equation}
\begin{remark}\label{rem:B2L}

 The reformulation \eqref{eq:1d-modelPGR} of the upwinding PG discretization \eqref{eq:1d-modelPG} involves the same piecewise linear test and trial space.% and allows for comparison with upwinding FD discretization method.
 % and  for new error analysis using an optimal (discrete) test norm.  
 %The idea was also used in \cite{CB2} for the special case of 1D Laplace equation. 
The coefficient of $(u'_h, w'_h)$, (that we call diffusion coefficient) in \eqref{eq:1d-modelPGR} is  $d=d_{\varepsilon,h} = \varepsilon +h\, b_1$.
Thus, the  left hand side of \eqref{eq:1d-modelPGR}  is given by the bilinear form  $b_d(u_h,w_h)$ defined only for continuous piecewise linear functions.  
 %upwinded coefficient, introduced in Section \ref{sec:optimalnorm-linears}, for the bilinear form defined only on continuous piecewise linear functions is %on how to use the same matrix idea for  (2.6) and (4.14) to get varational formulations with the the same bilinear form (on linears) and compare the (different) RHS functionals) 
Consequently, for the given test space $V_h$,  the optimal test norm  on $\M_h$ is 
 \begin{equation}\label{eq:COptNorm-hB} 
\|u_h\|_{*,h}^2  =( \varepsilon +h\, b_1)^2\, |u_h|^2 + |u_h|^2_{*,h}, 
\end{equation}
where  $|u_h|^2_{*,h}$  is defined in \eqref{eq:PhTu-Norm}, see Section \ref{sec:optimalnorm-linears}. 
 \end{remark}
 
% NEED to see if we need  this next estimate... ON THE OLD version on December 18, 2023
 
%%%%%%%%%%%
The reformulation \eqref{eq:1d-modelPGR} leads to the  linear system 
\begin{equation}\label{1d-PG-ls}
\left ( \left (\frac{\varepsilon}{h} + b_1 \right ) S+ C \right )\, U = F_{PG}, 
\end{equation}
where \(U,F_{PG}\in\R^{n-1}\)  with:
 \[
 U:=\begin{bmatrix}u_1\\u_2\\\vdots\\u_{n-1}\end{bmatrix},\quad F_{PG}:= \begin{bmatrix}(f,\varphi_1)\\ (f,\varphi_2)\\\vdots \\ (f,\varphi_{n-1})\end{bmatrix} + 
 \begin{bmatrix} (f, B_1 -B_2) 
 \\  (f, B_2-B_3)\\ \vdots \\  (f, B_{n-1} -B_n) \end{bmatrix}, 
\]
and $S, C$ are the matrices defined in \eqref{eq:SC}, for the FD discretiztion.% of Section \ref{sec:FD}.a%%%%REMARK%%%
\begin{remark}\label{rem:Phi-B}
 Here, we note that by using  the notation
\[
 \ \ d_{\varepsilon,h} = \varepsilon +h\, b_1, \ \text{or} \  \ \frac{\varepsilon}{h} + b_1 = \frac{d_{\varepsilon,h} }{h},  
\]
the matrix of the finite element  system  \eqref{1d-PG-ls} is 
\begin{equation} \label{eq:M4CDfe} 
M_{fe}= tridiag\left ( -\frac{d_{\varepsilon,h}}{h} - \frac{1}{2},\  \frac{2\, d_{\varepsilon,h}}{h},\  -\frac{d_{\varepsilon,h}}{h} + \frac{1}{2} \right ).
\end{equation}

\end{remark}
%%%%%%%%%%%%%%%
%SUB-SECTION
%%%%%%%%%%%%%%%
\subsection{Comparing the upwinding  FD and the  PG FE methods } \label{sec:System-comparison} 
In this section, we compare the general upwinding Petrov-Galerkin finite element discretization \eqref{eq:1d-modelPG} with the general upwinding finite difference method \eqref{FDUWgeneral} with $\varepsilon_{h}= \varepsilon (1+\Phi(\P e)) $ as defined in Section \ref{sec:FD}.  We consider that the bubble $B$ for the PG method is chosen such that $\varepsilon_{h}= \varepsilon + h\, b_1$, i.e., \\
 \[
 \Phi(\P e)=\Phi \left(\frac{h}{2\varepsilon}\right) =2\, b_1\, \P e,
 \]
where $b_1$ is defined by \eqref{b1}. 
  In this case, $\varepsilon_{h}= d_{\varepsilon,h}$, and  a direct  consequence of the matrix formulas  \eqref{eq:M4CDfd} and  \eqref{eq:M4CDfe} is that the FD matrix of the system \eqref{1d-LSfd} coincides with the FE matrix of the system \eqref{1d-PG-ls}, i.e., 
 \begin{equation}\label{Mfd=Mfe}
 M_{fd} =M_{fe}. 
 \end{equation}
 
  Since $b_1$ is allowed to depend on $\varepsilon$ and  $h$,   the  function   $\Phi(\P e) =2\, b_1\, \P e$ can recover  the  functions $\Phi$ used for the classical upwinding finite difference methods. 
 Furthermore, the {\it artificial diffusion}, $\varepsilon \Phi(\P e) $ that is introduced for upwinding  FD discretization, becomes exactly the integral of a bubble function that defines the corresponding  upwinding PG method, i.e., 
  \[
 \varepsilon \Phi(\P e)  = h\, b_1=\int_0^h B.
 \]
 % that we use for a corresponding  FE upwinding PG method. % that is a global variational formulation.  %We will see in the next sections that this observation leads to further connections between the two methods.  
  
Next, we show that the FD system \eqref{1d-LSfd} and the FE system \eqref{1d-PG-ls} correspond to variational formulations that use the same bilinear form defined on $\M_h \times \M_h$ and compare the 
right hand side functionals. 

For a continuous function $\theta : [0,1]\to \R$ such that $\theta(0)=\theta(1)=0$, the composite trapezoid rule (CTR) on the uniformly distributed  nodes $x_0,x_1, \ldots, x_n$ is 
\begin{equation}\label{CTR}
\int_0^1\theta(x)dx \approx T_n(\theta):=h\, \sum_{i=1}^{n-1} \theta(x_i) .
\end{equation}

If the vector $u^{FD}=[u_1, \ldots, u_{n-1}]^T$ is the solution of the finite difference system \eqref{FDUWgeneral}, then we define the corresponding proxy function  $u_h^{FD} \in \M_h$ by
\[
u_h^{FD}:=\displaystyle \sum_{i=1}^{n-1} u_i\varphi_i. 
\] 
From Section \ref{sec:FD},  we have  that $u^{FD}$ is the solution of the system 
\begin{equation}\label{sys:Mfd}
M_{fd}\, u^{FD}=h\, F,  \ \text{where}  \ F=[f(x_1), \cdots, f(x_{n-1})]^T. 
\end{equation}
Elementary calculations show that  
\begin{equation}\label{RHSfe-fd}
h\, f(x_j)=T_n(f\varphi_j)=T_n(f(\varphi_j + B_j-B_{j+1})), \ \text{for} \  j=1,2,\cdots, n-1.
\end{equation}
The following remark illustrates another connection between the upwinding FD and the bubble PG FE method.
\begin{remark}\label{FDis FE} 
In light of \eqref{Mfd=Mfe} and \eqref{RHSfe-fd}, we note that %\eqref{uhFD}, \eqref{uhFE}  and the fact that 
%\[T_n \left (f \,  \sum_{i=1}^{n}   w'_h B_i \right )= \sum_{i=1}^{n} T_n (f\, w'_h B_i)=0, \]
 the upwinding FD system \eqref{FDUWgeneral} can be obtained from  the  PG discretization \eqref{eq:1d-modelPG} reformulated as \eqref{eq:1d-modelPGR}  and leading to the system \eqref{1d-PG-ls}, where the  entries of the RHS vector, $\int_0^1 f(\varphi_j + B_j-B_{j+1})$ are approximated  with  the  CTR approximations. 
 
 At the local level, this corresponds to  using  the standard trapezoid rule  to approximate each of the integrals
\begin{equation}\label{eq:3int}
\int_{x_{j-1}}^{x_j} f\, \varphi_j , \ \ \int_{x_{j}}^{x_{j+1}} f\, \varphi_j ,\  \text{and}\  \int_{x_{j-1}}^{x_j}  f\, B_j. 
\end{equation}
\end{remark}

Next, we will justify that, under the assumption $\Phi(\P e) =2\, b_1\, \P e$, the upwinding FD method  \eqref{FDUWgeneral}  and  the  PG discretization \eqref{eq:1d-modelPG} can be viewed as variational formulations 
using the same bilinear form defined on $\M_h\times \M_h$. First, for the FD formulation, using \eqref{RHSfe-fd},  the algebra of Section \ref{sec:1d-lin-discrete}, with $\varepsilon \to \varepsilon_h$, and the notation introduced in Section  \ref{sec:optimalnorm-linears}, we have that the system \eqref{sys:Mfd}  of the upwinding FD method  \eqref{FDUWgeneral} corresponds to the variational formulation 
\[
b_d(u_h^{FD}, \varphi_j)=T_n(f\, \varphi_j) = T_n(f(\varphi_j + B_j-B_{j+1})), \ \ \text{for all} \ j=1,2,\ldots, n-1,
\]
with  $d= \varepsilon_h =  d_{\varepsilon,h}= \varepsilon + h\, b_1$.  %$d=d_{\varepsilon,h} = \varepsilon +h\, b_1$. 
Based on the linearity of $b_d(u_h^{FD},\cdot)$ and $T_n(f,\cdot)$, we can  further conclude that  
\begin{equation}\label{uhFD}
b_d(u_h^{FD},w_h)=T_n(f\, w_h),  \ \text{for all} \ w_h \in \M_h.
\end{equation}

Second, for the FE formulation, if $u_h=u_h^{FE}$ is the solution of \eqref{eq:1d-modelPG}, using the equivalent  variational formulation \eqref{eq:1d-modelPGR} with  $d= d_{\varepsilon,h}= \varepsilon + h\, b_1$, we have 
\begin{equation}\label{uhFE}
b_d(u_h^{FE}, w_h)=(f, w_h )  + h\,  (f, w'_h  \sum_{i=1}^{n}   B_i),   \ \text{for all} \ w_h\in \M_h.
\end{equation}

Based on the reformulations \eqref{uhFD} and \eqref{uhFE}, we can estimate now the difference between the upwinding FD and the  bubble PG solutions in the optimal trial norm. 
If we define  the linear functionals $F_1, F_2:\M_h \to \R$ by 
\[
\begin{aligned}
F_2(w_h) & =(f\, w_h) + h\,  (f, w'_h  \sum_{i=1}^{n}   B_i), \\
F_1(w_h) & = T_n \left ( f\, w_h + h\,  f \, w'_h  \sum_{i=1}^{n}   B_i \right )=T_n(f\, w_h) , 
\end{aligned}
\]
then the functionals  $F_h, W_h:\M_h \to \R$ defined by 
\[
 F_h(w_h) =\int_0^1f\, w_h \ dx -T_n(f\, w_h), \ \text{and} \ 
   W_h(w_h)  =  h\,  (f, w'_h  \sum_{i=1}^{n}   B_i)
\]
are also  linear functionals and
\[
F_2(w_h) - F_1(w_h)= F_h(w_h) + W_h(w_h).
\]
  Using Remark \ref{rem:u2-u1}, and the norm  $\|\cdot \|_{*,h}$ described in \eqref{eq:COptNorm-hB}, we obtain 

\begin{equation}\label{Fh}
\begin{aligned}
\|u_h^{FE}-u_h^{FD}\|_{*,h} & =\|F_2 -F_1\|_{\M_h^*}=\|F_h+W_h\|_{\M_h^*} \\
& \leq \|F_h\|_{\M_h^*}  + \|W_h\|_{\M_h^*}. 
\end{aligned}
\end{equation}
%where the supremum is taken over all non-zero vectors. 
 It was proved in  \cite{CB2}, by using standard approximation properties of Trapezoid Rule, that if $f\in\C^2([0,1])$, then 
 \begin{equation}\label{eq:NormFh}
 \|F_h\|_{\M_h^*} \leq h^2\left( \frac{\|f''\|_{\infty}}{12}+\frac{\|f'\|_{\infty}}{6}\right).
\end{equation}
 On the other hand, using the Cauchy-Schwartz inequality and assuming that 
 $
 \|B\|_\infty \leq M 
 $, 
 it is easy to check that 
 \[
 \|W_h\|_{\M_h^*} \leq M\, h \|f\|_{L^2}.
 \]
Consequently, we obtain 
\begin{equation}\label{udiff}
\|u_h^{FE}-u_h^{FD}\|_{*,h} \leq h^2\left( \frac{\|f''\|_{\infty}}{12}+\frac{\|f'\|_{\infty}}{6}\right) + M\, h \|f\|_{L^2}.
\end{equation}
The next remark provides a way to improve a standard upwinding FD method by using its connections with a bubble PG method. 
\begin{remark} \label{rem:FE-FDnorm} 
The estimate \eqref{udiff} is suboptimal in the sense that the second term in the RHS of \eqref{udiff} is only $O(h)$. In light of  Remark \ref{FDis FE}, this can be improved  by using  a modified  upwinding  FD method as follows. We start with the  corresponding PG  discretization  \eqref{eq:1d-modelPG}, reformulated as \eqref{eq:1d-modelPGR}, where  the bubble $B$ is chosen such that $\varepsilon_{h}= \varepsilon + h\, b_1$ ( or $\Phi(\P e)=2\, b_1\P e$), leading to the system \eqref{1d-PG-ls}. Then  the RHS entries of \eqref{1d-PG-ls}, i.e.,   $\int_0^1 f(\varphi_j + B_j-B_{j+1})$, are  approximated  with  better (than trapezoid) quadratures.  %compute the the integrals $\int_0^1 f(\varphi_j + B_j-B_{j+1})$. 
This suggests in fact to use a higher order quadrature, e.g., Cavalieri-Simpson or Gaussian quadrature, to approximate each of the integrals in \eqref{eq:3int}. In this way, new upwinding FD methods can be obtained by better approximating the dual vector of upwinding PG methods. Besides an improvement of the order of the estimate \eqref{udiff},  numerical computations  show that the  bubble PG approximation is better than related upwinding FD methods even in the discrete infinity error.  
\end{remark}
Consequently, the upwinding FD method  \eqref{FDUWgeneral}  for solving \eqref{eq:1d-model} can be improved just be modifying the  RHS vector of the FD system \eqref{1d-LSfd}. The new  $j-$th  entry of the RHS system  \eqref{1d-LSfd} is obtained  by approximating  the integral  $(f,\varphi_j +B_j-B_{j+1}) $  using, locally, higher order quadratures. Since we can choose different bubble functions $B$ and different quadratures to locally approximate   
the dual vectors, the improvement process is not unique. 

On the other hand, any  bubble PG method uses a fixed quadrature to approximate the dual vector. Thus, numerically, this performs  identically with an upwinding FD method with a special RHS induced by the quadrature. 
%This will lead to different  evaluations of  $f$  at  points dictated by the choice of the  quadrature and the corresponding quadrature nodes for estimating  the integrals in \eqref{eq:3int}.

 %%%%%%%%%%%%%%%
%SUB-SECTION
%%%%%%%%%%%%%%%
 %%%%%%%%%%%%%%%%%%%%%%%%%%
\section{Upwinding PG with quadratic bubble functions} \label{sec:Quad-Bubbles} 
%%%%%%%%%%%%%%%%%%%%%%%%%%%%

In this section, we consider the model problem \eqref{VF1d}  with the discrete space 
 $\M_h= span\{ \varphi_j\}_{j = 1}^{n-1}$ and $V_h$ a modification of $\M_h$  using {\it quadratic bubble functions}.
 The resulting method can be found in e.g.,  \cite{CRD-results,  zienkiewicz76, mitchell-griffiths, ZGHupwind76, zienkiewicz2014}. However, based on the results of the previous section, we relate the {\it quadratic bubble PG} method with the general upwinding FD method and present  ways to improve the performance of upwinding FD method.
 
   First, for a parameter $\beta>0$,  we define the bubble function $B$ on $[0, h]$ by
 \[
 B(x)= \frac{4\, \beta}{h^2} x(h-x).
 \]
 Elementary calculations show that \eqref{b1} %and \eqref{b2} 
 holds with $b_1=\frac{2\, \beta}{3}$. % and $b_2=\frac{16 \beta^2} {3h}$, respectively. 
 Using the function $B$ and the general construction of Section \ref{sec:GenBubble}, we define the set of bubble functions $\{B_1, B_2, \cdots,B_n\}$  on $[0, 1]$ and 
 \[
 V_h:= span \{ \varphi_j + (B_{j}-B_{j+1}) \}_{j= 1}^{n-1}.
 \]
  In this case, we have $d_{\varepsilon,h} =\varepsilon + \frac{2\beta}{3} h$. According to \eqref{eq:M4CDfe},  we obtain 
 \[
M_{fe}  =  tridiag\left ( -\frac{\varepsilon}{h} - \frac{2\beta}{3} - \frac{1}{2},\  \frac{2\, \varepsilon}{h}+ \frac{4\beta}{3} ,\   -\frac{\varepsilon}{h} - \frac{2\beta}{3} + \frac{1}{2} \right ). 
\]
For $\varepsilon_h=d_{\varepsilon,h}= \varepsilon + \frac{2\beta}{3} h$, or $\Phi(\P e)=2\, b_1\P e =\frac{4\beta}{3} \P e$, as presented at the beginning of Section \ref{sec:System-comparison}, the matrix of the system \eqref{1d-LSfd} is $M_{fd}=M_{fe}$.  

Here, we note that  we can relate any upwinding FD method defined by an admissible function $\Phi(\cdot)$ to  
our quadratic bubble PG method introduced in this section. This is justified by the fact that we can choose $\beta$, hence $b_1=\frac{2\, \beta}{3}$, such that   $\varepsilon_h=\varepsilon (1+  \Phi(\P e))=\varepsilon + b_1\, h$, i.e.,  
\[
\beta= \frac{3}{4} \frac{\Phi(\P e)}{\P e}. 
\]
%we can still have $M_{fd}=M_{fe}$   by choosing the  upwi PG bubble withand improve the upwinding FD method
Now, we can  benefit from  Remark \ref{rem:FE-FDnorm}.  By using, for example, the Cavalieri-Simpson (CS) rule, we can improve  the upwinding FD method  \eqref{FDUWgeneral}  for solving \eqref{eq:1d-model}  with $\varepsilon_h= \varepsilon + \frac{2\beta}{3} h$ by modifying  only the  RHS vector of  the FD system \eqref{1d-LSfd}. The new  $j-$th  entry of the RHS system  \eqref{1d-LSfd} is obtained  by approximating  
$(f,\varphi_j +B_j-B_{j+1})$ using the CS rule
 \[
 \int_a^b g(x) \, dx \approx \frac{b-a}{6} \left ( g(a) + 4 g\left (\frac{a+b}{2}\right ) + g(b) \right ), 
 \]
 on each mesh interval.
 This leads to replacing $h f(x_j) $ in  \eqref{1d-LSfd} by the value 
 \[
\frac{h}{3}  \left [  (1+2\, \beta) f(x_j-h/2) + f(x_j) + 
  (1-2\, \beta) f(x_j+h/2) \right ].
 \]
 As a specific application, we consider the case $\beta=3/4$ that leads to  $d_{\varepsilon,h}=\varepsilon_h =\varepsilon +h/2$. In this case, the FE matrix of the system \eqref{1d-PG-ls} coincides with  the matrix of the standard upwinding  FD  system \eqref{FDUW1}, and we have 
  \[
M_{fe} = M_{fd} =  tridiag\left ( -\frac{\varepsilon}{h} - 1,\  \frac{2\, \varepsilon}{h}+1,\   -\frac{\varepsilon}{h} \right ).
\]
To improve the performance of the upwinding  FD  method \eqref{FDUW1}, we can consider the   {\it  CS quadratic upwinding  method}  that  solves  the system
 \[
 M_{fd}\, U = G, \ \text{where}
 \] 
\[
G_j= \frac{h}{3}  \left (\frac{5}{2}  f(x_j-h/2) + f(x_j) -  \frac{1}{2}  f(x_j+h/2) \right),\  j=1,2,\cdots,n-1.
\]

Numerical results show that the {\it  Cavalieri-Simpson FD  (CS-FD)  method} performs better than the standard upwinding  FD method even if we measure the  error in the discrete infinity norm. We note that the standard upwinding FD method is in fact  the Trapezoid Finite Difference (T-FD) as a result of the quadratic bubble PG method using theTrapezoid rule for estimating the dual vector. 

We solved \eqref{eq:1d-model} with $\kappa=1$ and $f(x) =2x$, which allows to find the exact solution and to compute the discrete infinity error approximation. For example, for  $\epsilon =10^{-6}$, the standard upwinding FD method, or T-FD  produces a discrete  infinity error of order $O(h)$, while the 
CS-FD method exhibits higher order using the same  discrete infinity error. For example, when $n=800$, the 
discrete infinity error for T-FD is $0.0012$, and for CS-FD, the error is $0.64 \times 10^{-6}$. For the two methods, the error behaviour and error order, in various norms, will be addressed in future work.  

%\begin{remark}\end{remark} 
%%%%%%%%%%%%%%%%%%%%%%%%%
%SUBSECTION
%%%%%%%%%%%%%%%%%%%%%%%%%%
\section{Upwinding PG with exponential bubble functions} \label{sec:Exponential-Bubbles} 
%%%%%%%%%%%%%%%%%%%%%%%%%%%%

We consider the model problem \eqref{VF1d}  with the discrete space \\ 
 $\M_h= span\{ \varphi_j\}_{j = 1}^{n-1}$ and $V_h$ a modification of $\M_h$ by  using an {\it exponential  bubble function}. We define the bubble function $B$ on $[0, h]$ as the solution of 
 \begin{equation}\label{eq:expB}
 -\varepsilon B'' -B' =1/h, \ B(0)=B(h)=0.
 \end{equation}
  Using the function $B$ and the general construction of Section \ref{sec:GenBubble}, we define the set of bubble functions $\{B_1, B_2, \cdots,B_n\}$  on $[0, 1]$ and 
 \begin{equation}\label{eq:VhE}
 V_h:= span \{ \varphi_j + (B_{j}-B_{j+1}) \}_{j= 1}^{n-1}=span \{ g_j \}_{j= 1}^{n-1},
  \end{equation}
where $ g_j:=\varphi_j + (B_{j}-B_{j+1})$, $j=1,2,\cdots,n-1$.

In order to address efficient computations of coefficients and the finite element matrix of the {\it exponential bubble PG method}, we introduce the following notation
 \begin{equation}\label{eq:g0}
 g_0:=\tanh(\P e)= \frac{ e^{\frac{h}{2\varepsilon}} -e^{-\frac{h}{2\varepsilon}}} {e^{\frac{h}{2\varepsilon}} +e^{-\frac{h}{2\varepsilon}}}= \frac  {1- e^{-\frac{h}{\varepsilon}} }{1+e^{-\frac{h}{\varepsilon}} },
  \end{equation}
 \begin{equation}\label{eq:ld}
 l_d:= \frac{1+g_0}{2}= \frac  {e^{\frac{h}{\varepsilon}}  -1}{e^{\frac{h}{\varepsilon}} +e^{-\frac{h}{\varepsilon}}}, \ \text{and} \ l_0:=\frac{1+g_0}{2 g_0} = \frac{l_d}{g_0}, 
  \end{equation}
   \begin{equation}\label{eq:ud}
 u_d:= \frac{1-g_0}{2}= \frac  {1- e^{-\frac{h}{\varepsilon}} }{e^{\frac{h}{\varepsilon}} +e^{-\frac{h}{\varepsilon}}}, \ \text{and} \ u_0:=\frac{1-g_0}{2 g_0} = \frac{u_d}{g_0}. 
  \end{equation}
It is easy to check that the unique solution of \eqref{eq:expB} is 
 \begin{equation}\label{expB}
 B(x)=l_0 \left (1 - e^{-\frac{x}{\varepsilon}} \right )- \frac{x}{h}, \ x \in [0, h],  \ \text{and} 
  \end{equation}
   \begin{equation}\label{IntexpB}
   \int_0^{h} B(x)\, dx = \frac{h}{2 g_0} - \varepsilon.
     \end{equation}
 
 %\subsubsection{Exponential Bubble Upwinding: Linear system and connections to FD upwinding system} \label{sec:Exponential-Bubbles-system}     
Consequently,  we have that \eqref{b1} holds with $b_1= \frac{1}{2 g_0} - \frac{\varepsilon}{h}$, and we obtain that   $\varepsilon +b_1\, h= \frac{h}{2 g_0}$. Using Remark \ref{rem:B2L},  the optimal test norm  on $\M_h$ is \begin{equation}\label{eq:COptNorm-hBe} 
\|u_h\|_{*,h}^2  =\left (\frac{h}{2 g_0} \right )^2\, |u_h|^2 + |u_h|^2_{*,h}
\end{equation}
where  $|u_h|^2_{*,h}$  is defined in \eqref{eq:PhTu-Norm}. 
  
  Using Remark \ref{rem:Phi-B}, we obtain that $ \frac{\varepsilon_{b,h}}{h} = \frac{1}{2 g_0}$, and the matrix for the PG finite element discretization with exponential bubble test space becomes
  \begin{equation} \label{eq:M4CDfeEB}
  \begin{aligned}
M_{fe}^e &= tridiag\left ( -\frac{1+g_0}{2g_0} ,\  \frac{1}{g_0},\   -\frac{1-g_0}{2g_0} \right ) \\ 
 &  =  tridiag \left (  -l_0,  {1}/{g_0}, -u_0 \right )= \frac{1}{g_0}\,   tridiag \left (  -l_d, 1, -u_d \right ).
 \end{aligned}
\end{equation}
For $\varepsilon_h=d_{\varepsilon,h}= \varepsilon + b_1h = \frac{h}{g_0} $, we have  $\Phi(\P e)=2\, b_1\P e =\P e \, \coth(\P e) -1$.  According to Section \ref{sec:System-comparison}, the matrix of the system \eqref{1d-LSfd}
is $M^e_{fd}=M^e_{fe}$. 

By applying Remark \ref{FDis FE}, using the trapezoid rule to approximate the dual vector for the PG method, we get the upwinding FD method known as the {\it Il'in-Allen-Southwell (IAS) method}, according to  \cite{roos-stynes-tobiska-96}, or to  the {\it Scharfetter-Gummel (SG) method}, according to \cite{quarteroni-sacco-saleri07}.

Using again Remark \ref{rem:FE-FDnorm}, and  the Cavalieri-Simpson rule, we can improve  the upwinding FD method  \eqref{FDUWgeneral}  for solving \eqref{eq:1d-model}  with $\varepsilon_h= \frac{h}{g_0} $ (which is the  IAS or the SG method), by modifying  only the  RHS vector of  the FD system \eqref{1d-LSfd}. The new  $j-$th  entry of the RHS system  \eqref{1d-LSfd} is obtained  by approximating     $(f,\varphi_j +B_j-B_{j+1}) $  using the Cavalieri-Simpson rule, on each mesh interval. Since the $B_j$ functions are generated by the exponential  bubble function $B$ given by \eqref{expB}, this leads to replacing $h f(x_j) $ in  \eqref{1d-LSfd}  for  $ j=1,2,\cdots,n-1$, with  
 \[
G_j:= \frac{h}{3}  \left [  (1+2\, B(h/2)) f(x_j-h/2) + f(x_j) + 
  (1-2\, B(h/2)) f(x_j+h/2) \right ].
 \]
 The  {\it  CS exponential upwinding  method}  reduces to solving  for  $U$ the system
 \[
 M^e_{fd}\, U = G. 
 \] 
 
As in the  polynomial bubble case, the  CS-FD method performs better than the standard IAS or SG method. 
It is important to mention here that the upwinding PG method based on the exponential bubble produces in fact the exact solution at the nodes, provided that the dual vector is computed exactly. 
% \subsubsection{Exponential Bubble Upwinding: Finite Element solutions is the  interpolant of the exact solution} \label{sec:Exponential-Bubbles-FEisInterp}    
Variants of this result seem to be known in various forms, see e.g., \cite{roos-stynes-tobiska-96, Roos-Schopf15}. We include a simple proof of the above statement that is based on the properties of the exponential bubble function $B$. % and the way we defined our test space $V_h$. 
In order to proceed with the proof, we will need to emphasize a few  properties of  the test functions $g_j$ as follows. For any $j=1,2, \cdots, n$, we have   
\begin{equation}\label{eq:gj}
 g_j =  \left\{
 \begin{array}{ccl}
    \displaystyle  B_j +   \varphi_j  & \mbox{if } \  x \in [x_{j-1}, x_j],\\ \\
    \displaystyle  - B_{j+1} +   \varphi_j  & \mbox{if } \  x \in [x_j, x_{j+1}],
\end{array} 
\right.\   \text{and} 
\end{equation}

\begin{equation} \label{eq:gjD}
 g_j^{'} =  \left\{
 \begin{array}{ccl}
    \displaystyle  B_j^{'}+ \frac{1}{h}  & \mbox{if } \  x \in (x_{j-1}, x_j),\\ \\
    \displaystyle  - B_{j+1}^{'} -   \frac{1}{h}  & \mbox{if } \  x \in (x_j, x_{j+1}).
\end{array} 
\right.\   
g_j^{''} =  \left\{
 \begin{array}{ccl}
     \displaystyle  B_j^{''}  & \mbox{if } \  x \in (x_{j-1}, x_j),\\ \\
    \displaystyle  - B_{j+1}^{''}  & \mbox{if } \  x \in (x_j, x_{j+1}).
\end{array} 
\right.
\end{equation}

Using \eqref{eq:gjD}  and the fact that on $[x_{j-1}, x_j]$ the functions $B_j$ satisfy the same differential  equations as $B$, see \eqref{eq:expB}, we obtain 
\begin{equation} \label{eq:gjEq}
-\varepsilon g_j'' -g_j' =0, \ \text{on} \ (x_{j-1}, x_j)\cup (x_j, x_{j+1}), 
\end{equation}
\begin{equation} \label{eq:gjJump}
g_j' (x_{j} -) - g_j' (x_{j} +) =B'(h) + B'(0) - \frac{2}{h} = \frac{1} {\varepsilon\, g_0}, 
\end{equation}

\begin{equation} \label{eq:gjm1}
g_j' (x_{j-1})=g_j' (x_{j-1} +) = B'(0) +  \frac{1}{h} = \frac{1} {\varepsilon} \, \frac{l_d} {g_0},  \ \text{and}
\end{equation}

\begin{equation} \label{eq:gjp1}
g_j' (x_{j+1})=g_j' (x_{j+1} -) = -B'(h) -  \frac{1}{h} = - \frac{1} {\varepsilon}\,  \frac{u_d} {g_0}.
\end{equation}
Next, we are ready to prove the following result.
\begin{theorem}\label{uFF} Let $\displaystyle u_h := \sum_{i=1}^{n-1} u_i \varphi_i$  be the finite element solution of  \eqref{eq:1d-modelPG}  with the test space as defined in \eqref{eq:VhE}. Then $u_h$ coincides with the linear interpolant $I_h(u)$ of  the exact solution $u$ of \eqref{eq:1d-model}, with   $\kappa=1$ on the nodes $x_0,x_1, \ldots, x_n$.   In other words, $u_j=u(x_j)$, $j=1,2,\cdots,n-1$.
\end{theorem}
\begin{proof} For any fixed  $j \in \{1,2,\cdots,n-1\}$, we multiply the differential equation  \eqref{eq:1d-model}  (with $\kappa=1$) by $g_j$   and integrate by parts on the interval $[x_{j-1}, x_{j+1}]$ to obtain

\begin{equation} \label{eq:IPde}
\varepsilon \, \int_{x_{j-1}}^{x_{j+1}} u' g'_j -  \int_{x_{j-1}}^{x_{j+1}} u g'_j  + (u g_j) |_{x_{j-1}}^{x_{j+1}}=  \int_{x_{j-1}}^{x_{j+1}}\, f g_j=(f, g_j).
\end{equation}
Using that $g_j(x_{j-1})=g_j(x_{j+1})=0$, the third term in the LHS of  \eqref{eq:IPde} is zero.
Next, we apply integration  by parts for both integrals in the LHS of \eqref{eq:IPde}, spliting the integration on two 
subintervals $[x_{j-1}, x_j]$ and $[x_j, x_{j+1}]$, such that $g_j$ is smooth enough,  to obtain

\begin{equation} \label{eq:IPde2}
\begin{aligned}
& \int_{x_{j-1}}^{x_{j}}(-\varepsilon g_j^{''} -g'_j)\, u + \int_{x_{j}}^{x_{j+1}}(-\varepsilon g_j^{''} -g'_j)\, u \\
 & + \varepsilon\,  (u\,  g'_j) |_{x_{j-1}}^{x_j} + \varepsilon\, ( u\,  g'_j) |_{x_{j}}^{x_{j+1}} = (f, g_j).
 \end{aligned}
\end{equation}
By using  \eqref{eq:gjEq},  from \eqref{eq:IPde2} we get
\begin{equation} \label{eq:IPde3}
\begin{aligned}
& -\varepsilon g_j' (x_{j-1})\,  u(x_{j-1}) + \varepsilon [ g_j' (x_{j} -) - g_j' (x_{j} +))\,  u(x_{j}] \\
 &-\varepsilon g_j' (x_{j+1})\,   u(x_{j+1}) = (f, g_j).
 \end{aligned}
\end{equation}
Combining \eqref{eq:IPde3} with \eqref{eq:gjJump}-\eqref{eq:gjp1},  we obtain
\begin{equation} \label{eq:IPde4}
-\frac{l_d}{g_0} \,  u(x_{j-1}) + \frac{1}{g_0} \,  u(x_{j})  - \frac{u_d}{g_0} \,  u(x_{j+1})=(f, g_j),\ j=1,\cdots,n-1.
\end{equation}
Here, we notice that the matrix of the system \eqref{eq:IPde4}  with vector unknown $U_e=[u(x_1),\cdots u(x_{n-1})]^T$,
coincides with the matrix  $M_{fe}^e$ (see \eqref{eq:M4CDfeEB})  of the system solving for the finite element solution of \eqref{eq:1d-modelPG} with exponential bubble test space. Since the right hand sides of the two systems are the same and equal to $[(f,g_1), \cdots, (f,g_{n-1}) ]^T$, and  $M_{fe}^e$ is invertible, we conclude that $u_j=u(x_j)$, $j=1,2,\cdots,n-1$.
\end{proof}

 As a consequence  of Theorem \ref{uFF},  using Remark \ref{rem:FE-FDnorm} and a similar technique used for obtaining  \eqref{eq:NormFh}, we state  the following result. %see \cite{CB2},
\begin{theorem} \label{th:EFd-I}
Let $f \in \C^{(m+1)}([0, 1])$ and assume that an  upwinding  FD  method is obtained from the exponential PG FE method by using a quadrature of order $O(h^{m+1})$, on each mesh interval,  to  approximate the dual vector $[(f,g_1), \cdots, (f,g_{n-1}) ]^T$.  Then 
\begin{equation}\label{udiffm}
\|I_h(u)-u_h^{FD}\|_{*,h} \leq \, C\, h^m, 
\end{equation}
where $C$ is a constant that depends on $f, B$, and their derivatives. 
\end{theorem}
We note that $|v_h(x)| \leq  |v_h|$ for all $x \in [0, 1]$ and all $v_h \in \M_h$,  and from the representation \eqref{eq:COptNorm-hBe}  of the discrete optimal norm $\|\cdot\|_{*,h}$, we have 
\[
\|v_h\|_{\infty,h}:= \max_{i=\overline{0,n}}|v_h(x_i)| \leq |v_h| \leq \frac{2g_0}{h} \|v_h\|_{*,h}, \ \text{for all} \ v_h\in \M_h.
\]

As a consequence of the Theorem \ref{th:EFd-I}  and the above estimate, we obtain the following result.

\begin{corollary}
Under the assumptions of Theorem \ref{th:EFd-I} we have
\begin{equation}\label{udiffm2}
\|I_h(u)-u_h^{FD}\|_{\infty,h} \leq \, 2g_0\, C\, h^{m-1},
\end{equation}
where $C$ is the constant used for  \eqref{udiffm}.
\end{corollary}

%
%\begin{remark}\label{rem:warning}
Numerical tests show that the estimate  \eqref{udiffm2} does  not hold  if  $\varepsilon <<h$, and $u_h^{FD}$ is replaced by  the  computed solution $u_{h,c}^{FD}$. %More precisely, by increasing the order of the quadrature in approximating the dual vector, does not lead to a significant decrease in the infinity error, for example  for the cases $\frac{\varepsilon}{h} \approx 0$. 
This can be justified based on the  error in computing $e^{-\frac{h}{\varepsilon}}$. We note that if 
$\frac{h}{\varepsilon}$ is too large, then $e^{-\frac{h}{\varepsilon}}$ is computed as $0$. For example, for the double precision arithmetic, we have that $e^{-36.05}$ is smaller than the $\epsilon$-machine. Thus,  $1+e^{-\frac{h}{\varepsilon}}$ is computed as $1$  for $\frac{h}{\varepsilon} \geq 36.05$. Using standard calculus limits, we have 
 that for $\frac{\varepsilon}{h} \to 0$, 
\[ 
g_0 \to 1, \ \text{and} \  g_j =\varphi_j +B_j-B_{j+1} \to \chi_{|_{[x_{j-1}, x_j]}}. \ \text{Consequently,  for} \ \frac{\varepsilon}{h} \to 0
\]
%Consequently,  for $\frac{\varepsilon}{h} \to 0$, 
\[
M_{fe}^e \to  tridiag(-1, 1, 0), \ \text{and} \ (f,g_j) \to  \int_{x_{j-1}}^{x_j} f(x) \, dx. 
\]
Based on these observations,  the computed  matrix $M_{fe}^e$ becomes \\  $tridiag(-1, 1, 0)$, if $\varepsilon <<h$. 
Using a high order quadrature  to estimate the dual vector of the exponential bubble PG method leading  to an upwinding FD method, we can get a very  accurate approximation of  $(f,g_j) \approx \int_{x_{j-1}}^{x_j} f(x) \, dx$,  especially  if $f$ is for example a polynomial function. 
Thus, the  computed linear system  is  very close or identical to the system 
\[
 [tridiag(-1, 1, 0)] \, U = \left [ \int_{x_{0}}^{x_1} f(x) \, dx, \cdots,  \int_{x_{n-2}}^{x_{n-1}} f(x) \, dx \right ]^T.
\]
The system  can be solved exactly to obtain
\[
u_j = \int_{0}^{x_j} f(x) \, dx, \ j=1,2,\cdots,n-1. 
\]
This implies that, when $\varepsilon <<h$,  the component $u_j$  of the computed PG discrete solution is very close or identical to the value $w(x_j)$ where $w(x)=  \int_{0}^{x} f(t) \, dt$. By decreasing $h$, as long as  $ \frac{\varepsilon}{h}$ is still very small, the computed discrete solution remains close  to the interpolant of $w$ on $[0, x_{n-1}]$. Thus, by taking the discrete infinity norm on the nodes $x_1,\cdots,x_{n-1}$,  we have
\[
\|I_h(u)-u_{h,c}^{FD}\|_{\infty,h} \approx \|I_h(u)-I_h(w)\|_{\infty,h}= \| I_h(u-w) \|_{\infty,h},  
\]
and the difference  $ \| I_h(u-w) \|_{\infty,h}$  approaches $\|(u -w)_{|_{[x_1, x_{n-1}]}} \|_{\infty}$ for $h\to 0$. % is  less sensitive to $h\to 0$. 

Consequently, the error $\|I_h(u)-u_{h,c}^{FD}\|_{\infty,h}$ is less sensitive to changes in $h\to 0$, as long as $\frac{\varepsilon}{h}$ is very small, and \eqref{udiffm2}  cannot be checked numerically. 
%\end{remark}

We discretized \eqref{eq:1d-model} for $\kappa=1$, $f(x) =2x$, $\epsilon =10^{-6}$, and for the exponential bubble PG method. We used the Gaussian quadrature  $G_3$, with three nodes to locally approximate the dual vector. For all  values of \\ $h= \frac{1}{100},  \frac{1}{200},  \frac{1}{400},  \frac{1}{800},  \frac{1}{1600}$, we obtained
\[
\|I_h(u)-u_{h,c}^{FD}\|_{\infty,h} \approx 2 \times 10^{-6}.
\]
In conclusion, if the upwinding PG with exponential bubble, or its FD versions are implemented, then we should avoid choosing  $h >>\varepsilon$. For a given $\varepsilon$, we can choose for example $h \leq 30 \varepsilon$. %Numerical experiments for  $f(x) =2x$ show that the order of the discrete infinity error could be even higher than the order suggested in  \eqref{udiffm2}, provided that we keep  $h$ small enough, e.g.,  $h \in (0, 30 \varepsilon )$.% ($h \leq 30 \varepsilon$).

 \bibliographystyle{plain} 
% \bibliography{CollectiveBib}
 
\def\cprime{$'$} \def\ocirc#1{\ifmmode\setbox0=\hbox{$#1$}\dimen0=\ht0
  \advance\dimen0 by1pt\rlap{\hbox to\wd0{\hss\raise\dimen0
  \hbox{\hskip.2em$\scriptscriptstyle\circ$}\hss}}#1\else {\accent"17 #1}\fi}
  \def\cprime{$'$} \def\ocirc#1{\ifmmode\setbox0=\hbox{$#1$}\dimen0=\ht0
  \advance\dimen0 by1pt\rlap{\hbox to\wd0{\hss\raise\dimen0
  \hbox{\hskip.2em$\scriptscriptstyle\circ$}\hss}}#1\else {\accent"17 #1}\fi}

%%%%%%%%%%%%%%%%%%%%
 \end{document}